\documentclass[12pt]{amsart}
\usepackage[colorlinks=true,urlcolor=blue, citecolor=red,linkcolor=blue,linktocpage,pdfpagelabels, bookmarksnumbered,bookmarksopen]{hyperref}
\usepackage[hyperpageref]{backref}
\usepackage{cleveref}
\usepackage{xcolor}
\usepackage{amsthm} 
\usepackage{latexsym,amsmath,amssymb}
\usepackage{accents}
\usepackage[colorinlistoftodos,prependcaption,textsize=tiny]{todonotes}
\usepackage{a4wide}
\usepackage{soul}
\usepackage{mathtools} 
\usepackage{xparse} 
  
\title[Existence of degree 1 minimizers]{Existence of minimal maps of degree one in $W^{\frac1p,p}(\S^1,\S^1)$ for $p \in [p',2]$, where $p' \approx 1.13924$}

\date{\today}

\author{Tomasz Kostrzewa}
\address[Tomasz Kostrzewa]{
Sorbonne Faculty of Science and Engineering,
Sorbonne University,
4 Pl. Jussieu,
75005, Paris, France
}
\email{tomasz.kostrzewa@etu.sorbonne-universite.fr}

\author{Katarzyna Mazowiecka}
\address[Katarzyna Mazowiecka]{
Institute of Mathematics, %
University of Warsaw,
Banacha 2,
02-097 Warszawa, Poland}
\email{k.mazowiecka@mimuw.edu.pl}
\definecolor{indigo}{rgb}{0.29, 0.0, 0.51}

plus 6pt minus 12pt
\newcommand{\dif}{\,\mathrm{d}}

\def\S{{\mathbb S}}

\newtheorem{theorem}{Theorem}
\newtheorem{lemma}[theorem]{Lemma}

\newcommand{\dd}{\,\mathrm{d}}
\newcommand{\dx}{\dd x}
\newcommand{\dy}{\dd y}

\newcommand{\R}{\mathbb{R}}

\newcommand{\Z}{\mathbb{Z}}

\newcommand{\brac}[1]{\left (#1 \right )}
\newcommand{\abs}[1]{\left |#1 \right |}

\newcommand{\barint}{
\rule[.036in]{.12in}{.009in}\kern-.16in \displaystyle\int }

\newcommand{\barcal}{\mbox{$ \rule[.036in]{.11in}{.007in}\kern-.128in\int $}}

\def\mvint_#1{\mathchoice
          {\mathop{\vrule width 6pt height 3 pt depth -2.5pt
                  \kern -8pt \intop}\nolimits_{\kern -3pt #1}}%
          {\mathop{\vrule width 5pt height 3 pt depth -2.6pt
                  \kern -6pt \intop}\nolimits_{#1}}%
          {\mathop{\vrule width 5pt height 3 pt depth -2.6pt
                  \kern -6pt \intop}\nolimits_{#1}}%
          {\mathop{\vrule width 5pt height 3 pt depth -2.6pt
                  \kern -6pt \intop}\nolimits_{#1}}}

\numberwithin{theorem}{section} \numberwithin{equation}{section}

\newcommand{\aleq}{\precsim}

\let\latexchi\chi
\makeatletter
\renewcommand\chi{\@ifnextchar_\sub@chi\latexchi}
\newcommand{\sub@chi}[2]{
  \@ifnextchar^{\subsup@chi{#2}}{\latexchi^{}_{#2}}%
}
\newcommand{\subsup@chi}[3]{
  \latexchi_{#1}^{#3}%
}
\makeatother

\begin{document}

\begin{abstract}
In this note, we show how the results of \cite{mazowiecka2020minimal}, combined with those of \cite{BBM}, imply the existence of minimal maps of degree one in \( W^{\frac{1}{p},p}(\mathbb{S}^1,\mathbb{S}^1) \) for \( p \in [p', 2] \), where \( p' \approx 1.13924 \). This provides an affirmative answer in this range to a question posed by Mironescu \cite{M07} and Brezis--Mironescu \cite{BM-book}. In order to do so, we complement the results of \cite{mazowiecka2020minimal} by extending them to the case \( n = 1 \) and \( 1 < p < 2 \), which had been excluded there for technical reasons.
\end{abstract}

\maketitle
\sloppy

\section{Introduction}
In \cite{M07}, Mironescu asked whether, for \( p > 1 \), minimizers of the energy
\[
 E_{\frac1p,p}(u,\mathbb{S}^1) \coloneqq \int_{\mathbb{S}^1} \int_{\mathbb{S}^1} \frac{|u(x) - u(y)|^p}{|x - y|^2} \, \mathrm{d}x \, \mathrm{d}y
\]
exist among maps \( u \in W^{\frac{1}{p},p}(\mathbb{S}^1,\mathbb{S}^1) \) of degree one; see also \cite[Open Problem 24]{BM-book}. The answer is affirmative when \( p = 2 \):
\begin{equation}\label{eq:minimalwenergyW122}
 \inf_{\substack{u \in W^{\frac12,2}(\mathbb{S}^1,\mathbb{S}^1)\\ \deg u = 1}} E_{\frac12,2}(u) = 4\pi^2,
\end{equation}
and the infimum is attained by Möbius transformations (see, e.g., \cite[Section 2.3]{M07}).

Later, in \cite{Mironescu}, Mironescu proved that there exists \( \varepsilon > 0 \) such that minimizers of \( E_{\frac{1}{p},p} \) exist for maps of degree one in \( W^{\frac{1}{p},p}(\mathbb{S}^1,\mathbb{S}^1) \) for all \( p \in (2 - \varepsilon, 2] \). This result was further improved in \cite{Mazowiecka_Schikorra_2024}, where the authors showed that such minimizers exist for \( p \in (2 - \varepsilon, 2 + \varepsilon) \) for a computable, small \( \varepsilon > 0 \).

The question of existence is related to the sharp constant in the celebrated estimate by Bourgain--Brezis--Mironescu \cite[Theorem 0.6]{BBM}, which states that for any \( u \in W^{\frac{n}{p},p}(\mathbb{S}^n,\mathbb{S}^n) \),
\begin{equation}\label{eq:BBM}
 |\deg u| \leq C_{n,p} [u]_{W^{\frac{n}{p},p}}^p.
\end{equation}
The optimal constant is known only in the case \( n = 1 \), \( p = 2 \), where \( C_{1,2} = \frac{1}{4\pi^2} \); see \cite[Point 7, p. 1090]{Mironescu}.

In this short note, we show that combining the results of \cite{mazowiecka2020minimal} with \eqref{eq:BBM}, along with estimates for the energy \( [\mathrm{Id}]_{W^{\frac{1}{p},p}(\mathbb{S}^1,\mathbb{S}^1)} \), leads to an existence result for \( p \in [p', 2] \), where \( p' \approx 1.13924 \).

\begin{theorem}\label{th:main}
Assume \( p \in [p', 2] \), where \( p' \approx 1.13924 \). Then
\[
\inf_{\substack{u \in W^{\frac{1}{p},p}(\mathbb{S}^1,\mathbb{S}^1)\\ \deg u = 1}} \int_{\mathbb{S}^1} \int_{\mathbb{S}^1} \frac{|u(x) - u(y)|^p}{|x - y|^2} \, \mathrm{d}x \, \mathrm{d}y
\]
is attained.
\end{theorem}

\medskip

\noindent\textbf{Plan of the paper.}
In Section~\ref{s:compliment}, we complement the results of \cite{mazowiecka2020minimal} by addressing the case \( n = 1 \) and \( p < 2 \), which was previously excluded due to a technical assumption concerning regularity. In Section~\ref{s:mainproof}, we provide the proof of Theorem~\ref{th:main}. The paper concludes with an appendix, in which we compute the energy \( E_{1/p,p}(\mathrm{Id}) \) of the identity map \( \mathrm{Id} \colon \mathbb{S}^1 \to \mathbb{S}^1 \), which yields the critical value \( p' \) for which the existence result holds.

{\bf Acknowledgment.} The project is co-financed by the Polish National Agency for Academic Exchange within Polish Returns Programme - BPN/PPO/2021/1/00019/U/00001.

\section{Complement to the results of \cite{mazowiecka2020minimal}}\label{s:compliment}

The main ingredient in our proof is Lemma 7.7 from \cite{mazowiecka2020minimal}. We begin by recalling the relevant notation. We denote by \( \pi_0 C^0(\mathbb{S}^n, \mathcal{N}) \) the set of free homotopy classes of \( C^0(\mathbb{S}^n, \mathcal{N}) \). Given a class \( \Gamma \in \pi_0 C^0(\mathbb{S}^n, \mathcal{N}) \), we define
\[
 \# \Gamma \coloneqq \inf_{u \in \Gamma \cap W^{\frac{n}{p},p}(\mathbb{S}^n, \mathcal{N})} E_{\frac{n}{p},p}(u, \mathbb{S}^n).
\]

\begin{lemma}[cf.~{\cite[Lemma 7.7]{mazowiecka2020minimal}}]\label{la:main}
Let \( p > 1 \). There exists \( \theta = \theta(p, n, \mathcal{N}) > 0 \) such that the following holds.

Let \( \Gamma_0 \in \pi_0 C^0(\mathbb{S}^n, \mathcal{N}) \). Then at least one of the following alternatives holds:
\begin{enumerate}
    \item There exists a minimizer of \( E_{\frac{n}{p},p}(\cdot, \mathbb{S}^n) \) in \( \Gamma_0 \).

    \item For every \( \delta > 0 \), there exist nontrivial free homotopy classes \( \Gamma_1 = \pi_1(\mathcal{N}) \gamma_1 \) and \( \Gamma_2 = \pi_1(\mathcal{N}) \gamma_2 \) such that
    \[
     \Gamma_0 = \pi_1(\mathcal{N}) \gamma_0 \subset \pi_1(\mathcal{N}) \gamma_1 + \pi_1(\mathcal{N}) \gamma_2,
    \]
    where \( \gamma_1, \gamma_2 \in \pi_n(\mathcal{N}) \), and the following estimates hold:
    \begin{align}
    \# \Gamma_1 + \# \Gamma_2 &\leq \# \Gamma_0 + \delta, \label{eq:sum}\\
    \theta < \# \Gamma_1 &< \# \Gamma_0 - \frac{\theta}{2}, \label{eq:gamma1}\\
    \theta < \# \Gamma_2 &< \# \Gamma_0 - \frac{\theta}{2}. \label{eq:gamma2}
    \end{align}
\end{enumerate}
\end{lemma}

Lemma~\ref{la:main} was proved in \cite{mazowiecka2020minimal} in the case \( n \geq 2 \), or when \( n = 1 \) and \( p \geq 2 \). The restriction \( p \geq 2 \) was a technical assumption introduced to avoid dealing with regularity issues for the fractional \( p \)-Laplacian when \( p < 2 \).

In this section, we prove Lemma~\ref{la:main} in the remaining case \( n = 1 \) and \( 1 < p < 2 \). The only missing ingredient in \cite{mazowiecka2020minimal} for treating this case was the regularity result \cite[Theorem 3.1]{mazowiecka2020minimal}. A careful analysis of its proof reveals that the assumption \( p \geq 2 \) was used in two places: in Step 2 on p.~26, and in the proof of Step 3, specifically in the inequality (3.29). In the following subsections, we show how to modify the argument in each case so that the proof works for \( n = 1 \) and \( p = \frac{1}{s} < 2 \).

\subsection{Adaptation of Step 2 in \cite[Proof of Theorem 3.1]{mazowiecka2020minimal} to the case \texorpdfstring{$p<2$}{p<2}}
\begin{theorem}\label{th:reg}
Let $1<p< 2$. Assume that $u \in W^{s,p}(\R^n)$, $f \in L^1(\R^n)$ solve
\begin{equation}\label{eq:blablabla}
 \int_{\R^n}\int_{\R^n} \frac{|u(x)-u(y)|^{p-2}(u(x)-u(y))(\varphi(x)-\varphi(y))}{|x-y|^{n+sp}} \dx \dy = \int_{\R^n} f \varphi \quad \forall \varphi \in C_c^\infty(B(R)).
\end{equation}
If $u \in L^\infty \cap C^\alpha(B(R))$ for some $\alpha > 0$ then $u \in W^{s+\gamma,p}(B(R/2))$ for any $\gamma < \min\{\frac{\alpha}{2p},1\}$.
\end{theorem}
\begin{proof}
 The first part of the proof remains the same. We recall the notation (taken originally from \cite{BL17}): for $f\colon \R^n \to \R^M$ and $h\in \R^n$ we write $f_h(x)\coloneqq f(x+h)$, $\delta_h f(x) \coloneqq f(x+h) - f(x)$, we also write $J_p(v) \coloneqq |v|^{p-2}v$.
 
 Arguing exactly as in \cite[Proof of Theorem 3.7]{mazowiecka2020minimal} we obtain
\begin{equation}\label{eq:fromsucks}
 \begin{split}
  \int_{B(R)}\int_{B(R)} & \eta(x) \frac{\brac{J_{p}(u_h(x)-u_h(y))-J_{p}(u(x)-u(y))} (u_h(x)-u_h(y) - (u(x)- u(y)))}{|x-y|^{n+sp}} \dx \dy\\
   &\aleq |h|^{\alpha} [u]_{C^{\alpha}}
	\left( ||f||_{L^1(\mathbb{R}^n)}
	+ [u]_{W^{s,p}(\mathbb{R}^n)}^{p-1} R^{\frac{n}{p} - s} \right),
  \end{split}
\end{equation}
where $\delta= \frac{1}{100}R$, $|h|<\delta$, $\eta \in C_c^\infty(B(R-20\delta),[0,1])$, $\eta \equiv 1$ in $B(R-30\delta)$, and $|\nabla \eta|\aleq \frac 1\delta$.

Indeed, for
\[
 \mathcal{I}_1 \coloneqq \int_{B(R)}\int_{B(R)} \frac{\brac{J_{p}(u_h(x)-u_h(y))-J_{p}(u(x)-u(y))} (\eta(x) \delta_h u(x)-\eta(y) \delta_h u(y))}{|x-y|^{n+sp}} \dx \dy
\]
we have by \cite[bottom of  p. 25, Proof of Theorem 3.7]{mazowiecka2020minimal} that
\begin{align*}
	\mathcal{I}_1
	& \ge
	\int_{B(R)} \int_{B(R)} \eta(x) \frac{(J_p(u_h(x) - u_h(y)) - J_p(u(x) - u(y))) (u_h(x) - u_h(y) - (u(x) - u(y)))}{|x-y|^{n+sp}}\dx\dy \\
	& - \int_{B(R)} \int_{B(R)} |\delta_h u(y)| \frac{|J_p(u_h(x) - u_h(y)) - J_p(u(x) - u(y))| |\eta(x) - \eta(y)|}{|x-y|^{n+sp}}\dx\dy.
\end{align*}
The second term can be bound from above by
\begin{align*}
	2 |h|^{\alpha} [u]_{C^{\alpha}} \delta^{-1} [u]_{W^{s,p}(B(R+\delta))}^{p-1} R^{\frac{n}{p} + (1-s)},
\end{align*}
which gives
\begin{align}\label{appendix_a_first}
	&
	\int_{B(R)} \int_{B(R)} \eta(x) \frac{(J_p(u_h(x) - u_h(y)) - J_p(u(x) - u(y))) (u_h(x) - u_h(y) - (u(x) - u(y)))}{|x-y|^{n+sp}}\dx\dy \\
	& \le \mathcal{I}_1 + 2 |h|^{\alpha} [u]_{C^{\alpha}} \delta^{-1} [u]_{W^{s,p}(B(R+\delta))}^{p-1} R^{\frac{n}{p} + (1-s)}. \notag
\end{align}
On the other hand for
\[
 \mathcal{I}_2 \coloneqq \left |\int_{\R^n\setminus B(R)}\int_{B(R-20\delta)} \frac{\brac{J_{p}(u_h(x)-u_h(y))-J_{p}(u(x)-u(y))} \eta(x) \delta_h u(x)}{|x-y|^{n+sp}} \dx \dy \right |
\]
we have by \cite[top of p. 25, Proof of Theorem 3.7]{mazowiecka2020minimal}
\begin{align}\label{appendix_a_second}
	\mathcal{I}_1 - 2 \mathcal{I}_2
    & \le
    \int_{B(R)} \int_{B(R)} \frac{(J_p(u_h(x) - u_h(y)) - J_p(u(x) - u(y))) (\eta(x) \delta_h u(x) - \eta(y) \delta_h u(y))}{|x-y|^{n+sp}}\dx\dy \\
	& \lesssim
	[u]_{C^{\alpha}(B(R))} ||f||_{L^1(\mathbb{R}^n)} |h|^{\alpha}. \notag
\end{align}
Combining \eqref{appendix_a_first} with \eqref{appendix_a_second} we get
\begin{align*}
	&
	\int_{B(R)} \int_{B(R)} \eta(x) \frac{(J_p(u_h(x) - u_h(y)) - J_p(u(x) - u(y))) \cdot (u_h(x) - u_h(y) - (u(x) - u(y)))}{|x-y|^{n+sp}}\dx\dy \\
	& \lesssim [u]_{C^{\alpha}(B(R))} ||f||_{L^1(\mathbb{R}^n)} |h|^{\alpha}
    + 2 \mathcal{I}_2
	+ 2 |h|^{\alpha} [u]_{C^{\alpha}} \delta^{-1} [u]_{W^{s,p}(B(R+\delta))}^{p-1} R^{\frac{n}{p} + (1-s)}.
\end{align*}
Now, as \cite[p. 25, Proof of Theorem 3.7]{mazowiecka2020minimal}
\begin{align*}
	\mathcal{I}_2
	\lesssim 2 |h|^{\alpha} [u]_{C^{\alpha}} [u]_{W^{s,p}(\mathbb{R}^n)}^{p-1} \delta^{-s} R^{\frac{n}{p}}
\end{align*}
we obtain
\begin{align*}
	&
	\int_{B(R)} \int_{B(R)} \eta(x) \frac{(J_p(u_h(x) - u_h(y)) - J_p(u(x) - u(y))) \cdot (u_h(x) - u_h(y) - (u(x) - u(y)))}{|x-y|^{n+sp}}\dx\dy \\
	& \lesssim
	|h|^{\alpha} [u]_{C^{\alpha}}
	\left( ||f||_{L^1(\mathbb{R}^n)}
	+ 4 [u]_{W^{s,p}(\mathbb{R}^n)}^{p-1} \delta^{-s} R^{\frac{n}{p}}
	+ 2 \delta^{-1} [u]_{W^{s,p}(B(R+\delta))}^{p-1} R^{\frac{n}{p} + (1-s)} \right) \\
	& \lesssim |h|^{\alpha} [u]_{C^{\alpha}}
	\left( ||f||_{L^1(\mathbb{R}^n)}
	+ [u]_{W^{s,p}(\mathbb{R}^n)}^{p-1} R^{\frac{n}{p} - s}
	+ [u]_{W^{s,p}(B(R+\delta))}^{p-1} R^{\frac{n}{p} - s} \right) \\
    & \lesssim |h|^{\alpha} [u]_{C^{\alpha}}
	\left( ||f||_{L^1(\mathbb{R}^n)}
	+ [u]_{W^{s,p}(\mathbb{R}^n)}^{p-1} R^{\frac{n}{p} - s} \right),
\end{align*}
which gives \eqref{eq:fromsucks}.

We will apply to the left-hand side of \eqref{eq:fromsucks} the following elementary inequality
\begin{equation}\label{eq:p-blabla}
 \langle|b|^{p-2}b - |a|^{p-2}a, b-a \rangle
 \ge (p-1)|b-a|^2 \int_0^1 |a+t(b-a)|^{p-2} \dif t,
\end{equation}
which is valid for $1<p<2$ and $a,b\in \R^M$ (see e.g., \cite[\textsection 12, (IV)]{Lindqvist}). We obtain
\[
 \begin{split}
  \int_{B(R-30\delta)}\int_{B(R-30\delta)}&\frac{|u_h(x) - u_h(y)-(u(x)-u(y))|^2 W^{p-2}(u,x,y,h)}{|x-y|^{n+sp}}\dif x \dif y\\
  &\aleq |h|^{\alpha} [u]_{C^{\alpha}}
	\left( ||f||_{L^1(\mathbb{R}^n)}
	+ [u]_{W^{s,p}(\mathbb{R}^n)}^{p-1} R^{\frac{n}{p} - s} \right),
 \end{split}
\]
where $W^{p-2}(u,x,y,h) = \int_{0}^{1} \abs{u(x) - u(y) + t(u_h(x) - u_h(y) - (u(x) - u(y)))}^{p-2} \dif t$.

Dividing both sides by $|h|^\alpha$ we get
\begin{equation}\label{eq:tralalalala}
 \begin{split}
  \int_{B(R-30\delta)}\int_{B(R-30\delta)}&\frac{\abs{\frac{u_h(x) - u_h(y)-(u(x)-u(y))}{|h|^{\frac{\alpha}{2}}}}^2W^{p-2}(u,x,y,h)}{|x-y|^{n+sp}}\dif x \dif y\\
  &\aleq [u]_{C^{\alpha}}
	\left( ||f||_{L^1(\mathbb{R}^n)}
	+ [u]_{W^{s,p}(\mathbb{R}^n)}^{p-1} R^{\frac{n}{p} - s} \right).
 \end{split}
\end{equation}
From Young's inequality we get for any $A>0$, $B>0$, $1<p<2$ 
\[
 B^{2-p}A^p\le \frac{A^2}{\frac2p} + \frac{B^2}{\frac{2}{2-p}} \le A^2 + B^2 \quad \Rightarrow \quad A^p \le A^2 B^{p-2} + B^p.
\]
Using the latter inequality in \eqref{eq:tralalalala} we obtain
\begin{equation}\label{eq:1}
 \begin{split}
  &\int_{B(R-30\delta)}\int_{B(R-30\delta)} \frac{\abs{|h|^{-\frac{\alpha}{2}}\brac{\delta_h u(x) - \delta_h u(y)}}^p}{|x-y|^{n+sp}}\dif x \dif y\\
  &\aleq [u]_{C^{\alpha}}
	\left( ||f||_{L^1(\mathbb{R}^n)}
	+ [u]_{W^{s,p}(\mathbb{R}^n)}^{p-1} R^{\frac{n}{p} - s} \right) \\
  &\quad + \int_{B(R-30\delta)}\int_{B(R-30\delta)} \frac{W^p(u,x,y,h)}{|x-y|^{n+sp}}\dif x \dif y.
 \end{split}
\end{equation}
Now we estimate
\begin{equation}\label{eq:2}
 \begin{split}
  &\int_{B(R-30\delta)}\int_{B(R-30\delta)} \frac{W^p(u,x,y,h)}{|x-y|^{n+sp}}\dif x \dif y\\
  &=\int_{B(R-30\delta)}\int_{B(R-30\delta)} \frac{\brac{\int_{0}^{1} \abs{u(x) - u(y) + t(u_h(x) - u_h(y) - (u(x) - u(y)))}^{p-2} \dif t}^{\frac{p}{p-2}}}{|x-y|^{n+sp}}\dif x \dif y\\
  &\le \int_{B(R-30\delta)}\int_{B(R-30\delta)} \frac{\max_{t\in[0,1]}\abs{u(x) - u(y) + t(u_h(x) - u_h(y) - (u(x) - u(y)))}^{p} }{|x-y|^{n+sp}}\dif x \dif y\\
  &\aleq \int_{B(R-30\delta)}\int_{B(R-30\delta)} \frac{|u(x) - u(y)|^p + |u(x+h) - u(y+h)|^p}{|x-y|^{n+sp}} \dif x \dif y\\
  &\aleq [u]^p_{W^{s,p}(B(R))}.
 \end{split}
\end{equation}
Combining \eqref{eq:1} and \eqref{eq:2} we obtain
\[
 [|h|^{-\frac \alpha2}\delta_h u]^p_{W^{s,p}(B(R-30\delta))} \le C(u, R, f,p, n, s) <\infty.
\]
The theorem follows from a difference quotient characterization of Sobolev spaces, see \cite[Lemma 3.8]{mazowiecka2020minimal} for details.
\end{proof}

\subsection{Step 3 in \cite[proof of Theorem 3.1]{mazowiecka2020minimal} for \texorpdfstring{$p<2$}{p<2}}
\hfill

In \cite[(3.29)]{mazowiecka2020minimal}, Hölder's inequality was used, but for $p = \frac{1}{s} < 2$, the exponent $\frac{1}{s} - 2$ is negative, and the argument must be changed. We therefore modify the estimate of $[u]_{W^{s_0,\frac{1}{s}}(B(r))}^{\frac{1}{s}}$ from \cite[p.~31]{mazowiecka2020minimal}.

We use the notation of the proof of Theorem 3.10 in \cite{mazowiecka2020minimal}. First we note that since $\sqrt{\eta}(x) \equiv 1$ on $B(r)$, we have for $x,y \in B(r)$,
\[
u(x) - u(y) = \tilde{\varphi}(x) - \tilde{\varphi}(y),
\]
where $\tilde{\varphi} := \sqrt{\eta}(u - (u)_{B(R)})$. Let us denote $\kappa \coloneqq \frac{2}{s} - 2$, which implies  $\frac 1s-\kappa>0.$ We have
\begin{equation}
 [u]_{W^{s_0,\frac{1}{s}}(B(r))}^{\frac{1}{s}} =
 \int_{B(r)}\int_{B(r)} \frac{|u(x)-u(y)|^{\frac{1}{s}-\kappa} |\tilde \varphi(x) - \tilde \varphi(y)|^\kappa}{|x-y|^{1+\frac{s_0}{s}}} \dx \dy.
\end{equation}

As in \cite[p.31]{mazowiecka2020minimal}
\begin{equation}
\begin{split}
|\tilde{\varphi}(x)-\tilde{\varphi}(y)|^2
 &= (\sqrt{\eta}(x)-\sqrt{\eta}(y)) (u(x)-(u)_{B(R)}) \cdot(\tilde{\varphi}(x)-\tilde{\varphi}(y))\\
 &\quad +  (u(x)-u(y))(\sqrt{\eta}(y)-\sqrt{\eta}(x))\tilde{\varphi}(x)\\
 &\quad +  (u(x)-u(y))(\varphi(x)-\varphi(y)).
 \end{split}
\end{equation}
Thus, using the subadditivity of $f(t)=t^{\frac{\kappa}{2}}$ and recalling that $|\nabla \sqrt{\eta}|\le \frac{C_1}{\delta}$ we get
\begin{equation}\label{eq:differenceofvarphitildetoalpha}
\begin{split}
|\tilde{\varphi}(x)-\tilde{\varphi}(y)|^\kappa
 & \le C_1^{\frac \kappa 2}\frac{|x-y|^{\frac \kappa 2}}{\delta^{\frac \kappa 2}}|u(x)-(u)_{B(R)}|^{\frac \kappa 2} \brac{|\tilde{\varphi}(x)-\tilde{\varphi}(y)|^{\frac \kappa 2}+|u(x)-u(y)|^{\frac \kappa 2}}\\
 &\quad +  \brac{(u(x)-u(y))(\varphi(x)-\varphi(y))}^{\frac \kappa 2}.
 \end{split}
\end{equation}
Moreover, we have
\begin{equation}
\begin{split}
 &|u(x)-u(y)|^{\frac{1}{s}-\kappa} \brac{(u(x)-u(y))\, (\varphi(x)-\varphi(y))}^{\frac \kappa 2}\\
 &=|u(x)-u(y)|^{\frac 1s\frac{2-\kappa}{2}}|u(x)-u(y)|^{\frac{1}{s}-\kappa-\frac 1s\frac{2-\kappa}{2}}\brac{(u(x)-u(y))\, (\varphi(x)-\varphi(y))}^{\frac \kappa 2}.
 \end{split}
\end{equation}
Applying to the latter Young's inequality with $r= \frac{2}{2-\kappa},\,r'=\frac{2}{\kappa}$,  we obtain
\begin{equation}\label{eq:Youngwitheps}
 \begin{split}
  &|u(x)-u(y)|^{\frac{1}{s}-\kappa} \brac{(u(x)-u(y))\, (\varphi(x)-\varphi(y))}^{\frac \kappa 2} \\
  &\le \frac{2-\kappa}{2} |u(x)-u(y)|^{\frac{1}{s}} + \frac{\kappa}{2} |u(x)-u(y)|^{\frac{1}{s}-2} (u(x)-u(y))\, (\varphi(x)-\varphi(y)).
 \end{split}
\end{equation}
Thus, combining \eqref{eq:differenceofvarphitildetoalpha} with \eqref{eq:Youngwitheps} we get
\begin{equation}
 \begin{split}
  |u(x) - u(y)|^\frac{1}{s} 
  &=  |u(x) - u(y)|^{\frac{1}{s}-\kappa}|\tilde \varphi(x) - \tilde \varphi(y)|^\kappa\\
  &\le  C_1^\frac \kappa 2 |u(x) - u(y)|^{\frac{1}{s}-\kappa} \brac{\frac{|x-y|^{\frac \kappa 2}}{\delta^{\frac \kappa 2}}|u(x)-(u)_{B(R)}|^{\frac \kappa 2} \brac{|\tilde{\varphi}(x)-\tilde{\varphi}(y)|^{\frac \kappa 2}+|u(x)-u(y)|^{\frac \kappa 2}}} \\
  &\quad + \frac{2-\kappa}{2} |u(x) - u(y)|^\frac{1}{s} + \frac{\kappa}{2} |u(x)-u(y)|^{\frac{1}{s}-2} (u(x)-u(y))\, (\varphi(x)-\varphi(y)).
 \end{split}
\end{equation}
Since $\frac{2-\kappa}{2}<1$ we may absorb one term and obtain
\begin{equation}
 \begin{split}
  |u(x) - u(y)|^\frac{1}{s}
  &\aleq  |u(x) - u(y)|^{\frac{1}{s}-\kappa} \brac{\frac{|x-y|^{\frac \kappa 2}}{\delta^{\frac \kappa 2}}|u(x)-(u)_{B(R)}|^{\frac \kappa 2} \brac{|\tilde{\varphi}(x)-\tilde{\varphi}(y)|^{\frac \kappa 2}+|u(x)-u(y)|^{\frac \kappa 2}}} \\
  &\quad +  |u(x)-u(y)|^{\frac{1}{s}-2} (u(x)-u(y))\, (\varphi(x)-\varphi(y)).
 \end{split}
\end{equation}

Hence,
\begin{equation}\label{eq:old(3.28)}
 \begin{split}
[u]_{W^{s_0,\frac{1}{s}}(B(r))}^{\frac{1}{s}}
&\aleq
 \int_{B(\rho - 3\delta)}\int_{B(\rho - 3\delta)} \frac{|u(x)-u(y)|^{\frac{1}{s}-2} \brac{(u(x)-u(y))\, (\varphi(x)-\varphi(y))}}{|x-y|^{1+\frac{s_0}{s}}} \dx \dy\\
&\quad + \delta^{-\frac \kappa 2} \int_{B(\rho - 3\delta)}\int_{B(\rho - 3\delta)} \frac{|u(x)-u(y)|^{\frac{1}{s}-\frac \kappa 2} |u(x)-(u)_{B(R)}|^{\frac \kappa 2}}{|x-y|^{1+\frac{s_0}{s}-\frac \kappa 2}} \dx \dy\\
&\quad + \delta^{-\frac \kappa 2} \int_{B(\rho - 3\delta)} \int_{B(\rho - 3\delta)} \frac{|u(x)-u(y)|^{\frac{1}{s}-\kappa} |\tilde{\varphi}(x)-\tilde{\varphi}(y)|^{\frac \kappa 2} |u(x)-(u)_{B(R)}|^{\frac \kappa 2} }{|x-y|^{1+\frac{s_0}{s}-\frac \kappa 2}} \dx \dy.
 \end{split}
\end{equation}
This is the inequality that we will replace (3.28) in \cite{mazowiecka2020minimal} with. The first term of (3.28) is the same term as in (3.28) and we may estimate the two second terms similarly as in (3.29), using H\"{o}lder's inequality.

For any $w\colon \R \to \R^M$ we have using that $\frac{1}{1-s\kappa}, \frac{2}{s\kappa}>1$ and $(1-s\kappa) + \frac{s\kappa}{2}+ \frac{s\kappa}{2}=1$
\begin{equation}\label{eq:old(3.29)}
 \begin{split}
  &\delta^{-\frac \kappa 2} \int_{B(\rho - 3\delta)} \int_{B(\rho - 3\delta)} \frac{|u(x)-u(y)|^{\frac{1}{s}-\kappa} |w(x)-w(y)|^{\frac \kappa 2} |u(x)-(u)_{B(R)}|^{\frac \kappa 2} }{|x-y|^{1+\frac{s_0}{s}-\frac \kappa 2}} \dx \dy\\
  &=\delta^{-\frac \kappa 2} \int_{B(\rho - 3\delta)} \int_{B(\rho - 3\delta)} \frac{|u(x)-u(y)|^{\frac{1}{s}-\kappa}}{|x-y|^{s(\frac 1s - \kappa)}}
  \frac{|w(x)-w(y)|^{\frac \kappa 2}}{|x-y|^{s\frac{\kappa}{2}}}
  \frac{  |u(x)-(u)_{B(R)}|^{\frac \kappa 2} }{|x-y|^{\frac{\kappa s^2 - \kappa s - 2s + 2s_0}{2s}
}} \frac{\dx \dy}{|x-y|}\\
  &\le \delta^{-\frac\kappa2}[u]_{W^{s,\frac 1s}(B(R))}^{\frac 1s - \kappa}[w]_{W^{s,\frac 1s}(B(R))}^{\frac \kappa 2} \brac{\int_{B(R)} \int_{B(R)} \frac{|u(x) - (u)_{B(R)}|^\frac{1}{s}}{|x-y|^{1+\frac{2}{s\kappa}\brac{\frac{\kappa s^2 - \kappa s - 2s + 2s_0}{2s}
}}}\dx \dy}^\frac{s\kappa}{2}\\
  &\aleq \delta^{-\frac\kappa2}[u]_{W^{s,\frac 1s}(B(R))}^{\frac 1s - \kappa}[w]_{W^{s,\frac 1s}(B(R))}^{\frac \kappa 2} \brac{\int_{B(R)}\frac{|u(x)-(u)_{B(R)}|^\frac 1s}{R^{\frac{2}{s\kappa}\brac{\frac{\kappa s^2 - \kappa s - 2s + 2s_0}{2s}
}}}\dx}^{\frac{s\kappa}{2}}\\
& \aleq \delta^{-\frac\kappa2}R^{\frac{\kappa}{2} - \frac{s_0 - s}{s}
}[u]_{W^{s,\frac 1s}(B(R))}^{\frac 1s - \frac \kappa2}[w]_{W^{s,\frac 1s}(B(R))}^{\frac \kappa 2}.
 \end{split}
\end{equation}
Here we need to take $s_0<s^2-s+1$ (close enough to $s$) to ensure $\frac{\kappa s^2 - \kappa s - 2s + 2s_0}{2s}<0$.

Applying \eqref{eq:old(3.29)} to the last two terms in \eqref{eq:old(3.28)} we obtain
\begin{equation}
 \begin{split}
  [u]_{W^{s_0,\frac{1}{s}}(B(r))}^{\frac{1}{s}}
&\aleq  \int_{B(\rho - 3\delta)}\int_{B(\rho - 3\delta)} \frac{|u(x)-u(y)|^{\frac{1}{s}-2} \brac{(u(x)-u(y))\, (\varphi(x)-\varphi(y))}}{|x-y|^{1+\frac{s_0}{s}}} \dx \dy\\
&\quad + \delta^{-\frac\kappa2}R^{\frac{\kappa}{2} - \frac{s_0 - s}{s}
}[u]_{W^{s,\frac 1s}(B(R))}^{\frac 1s}\brac{1+\brac{\frac{R}{\delta}}^{\frac{\kappa}{2}}}.
 \end{split}
\end{equation}

Now the final estimate of the proof on top of the page 39 in \cite{mazowiecka2020minimal} is exactly the same with the exception that the third term on the right-hand side after the first inequality sign needs to be replaced by
\[
 \delta^{-\frac\kappa2}R^{\frac{\kappa}{2} - \frac{s_0 - s}{s}
}[u]_{W^{s,\frac 1s}(B(R))}^{\frac 1s}\brac{1+\brac{\frac{R}{\delta}}^{\frac{\kappa}{2}}}.
\]
but we still have
\[
  \delta^{-\frac\kappa2}R^{\frac{\kappa}{2} - \frac{s_0 - s}{s}
}\brac{1+\brac{\frac{R}{\delta}}^{\frac{\kappa}{2}}} \le \delta^{-\frac{s_0-s}{s}}\brac{\frac{R}{\delta}}^\frac 1s.
\]
This finishes the proof.
\section{Proof of \texorpdfstring{\Cref{th:main}}{th:main}}\label{s:mainproof}
We are now ready to proceed with the proof of the main result.

\begin{proof}[Proof of \Cref{th:main}]
Assume the theorem is false and that there is no map of degree 1 in $W^{\frac 1p,p}(\S^1,\S^1)$ for which
\[
 \inf_{u\in W^{1/p,p}(\S^1,\S^1),\deg u =1}\left\{\int_{\S^1} \int_{\S^1} \frac{|u(x)-u(y)|^p}{|x-y|^{2}}  \dif x \dif y\right\}
\]
is achieved.

Let $\Gamma_1$ be the class of degree one maps from $\S^1$ to $\S^1$. 
Then by \Cref{la:main} we obtain the existence of two other homotopy classes, $\Gamma_{d_1}$ --- class of maps of degree $d_1$, where $0\neq d_1\in\Z$ and $\Gamma_{d_2}$ --- class of maps of degree $d_2$, where $0\neq d_2\in\Z$ such that
\[
 \Gamma_1 \subset \Gamma_{d_1} + \Gamma_{d_2} 
\]
and for any $\delta>0$ the following estimate holds
\begin{equation}\label{eq:suminpf}
\# \Gamma_{d_1} + \# \Gamma_{d_2} \le \# \Gamma_{1} + \delta.
\end{equation}
By \eqref{eq:gamma1} and \eqref{eq:gamma2} we know that $d_1, \, d_2 \notin\{-1,0,1\}$. Thus, 
\begin{equation}\label{eq:d1d2min}
 \min \brac{|d_1| + |d_2|} = 5.
\end{equation}

Recall from \cite[Point 7, p. 1090]{Mironescu}, that for any map $u\in W^{\frac12,2}(\S^1,\S^1)$ we have
\begin{equation}\label{eq:one}
 4\pi^2 |\deg u|\le [u]^2_{W^{\frac12,2}(\S^1,\S^1)}.
\end{equation}
Moreover, as in \cite[Proof of Theorem 2]{Mironescu} for any $1<p<2$ for any $u\in W^{\frac1p,p}(\S^1,\S^1)$ we have the inequality
\begin{equation}\label{eq:two}
\begin{split}
 [u]^2_{W^{\frac12,2}(\S^1,\S^1)} = \int_{\S^1} \int_{\S^1} \frac{|u(x) - u(y)|^p}{|x-y|^2}|u(x) - u(y)|^{2-p} \dif x \dif y\le 2^{2-p} [u]^p_{W^{\frac1p,p}(\S^1,\S^1)}.
\end{split}
 \end{equation}
Thus, combining \eqref{eq:one} and \eqref{eq:two} we get for any $u\in W^{\frac1p,p}(\S^1,\S^1)$
\[
 \frac{4\pi^2}{2^{2-p}}|\deg u|\le [u]^p_{W^{\frac 1p,p}(\S^1,\S^1)}.
\]
Therefore,
\begin{equation}\label{eq:diest}
 \frac{4\pi^2}{2^{2-p}}|d_i|\le \inf_{u_i\in W^{\frac 1p,p}(\S^1,\S^1), \deg u_i = d_i}[u_i]^p_{W^{\frac 1p,p}(\S^1,\S^1)} = \# \Gamma_{d_i}.
\end{equation}
Thus, using estimates \eqref{eq:d1d2min} and \eqref{eq:diest} we get
\begin{equation}\label{eq:LHSest}
 5\frac{4\pi^2}{2^{2-p}} \le \frac{4\pi^2}{2^{2-p}}(|d_1| + |d_2|)\le \# \Gamma_{d_1} + \# \Gamma_{d_2}.
\end{equation}
On the other hand, the map $\mathrm{Id}\colon \S^1\to \S^1$ is of degree one. Thus,
\begin{equation}\label{eq:RHSest}
\begin{split}
 \# \Gamma_1 \le E_{1/p,p}(\mathrm{Id},\S^1) = \int_{\S^1} \int_{\S^1} \frac{1}{|x-y|^{2-p}} \dif x \dif y.
%
\end{split}
\end{equation}
Combining \eqref{eq:suminpf} with \eqref{eq:LHSest} we get for any $\delta>0$
\[
 5\frac{4\pi^2}{2^{2-p}}\le \frac{4\pi^2}{2^{2-p}}(|d_1| + |d_2|) \le \# \Gamma_{d_1} + \#\Gamma_{d_2} \le \#\Gamma_1 +\delta \le E_{\frac1p,p}(\mathrm{Id},\S^1) + \delta.
\]
We have moreover (see Lemma \ref{decreasing_energy}) that for $1 < p_1 < p_2 < 2$ there is
\[
  E_{\frac{1}{p_1},p_1}(\mathrm{Id},\S^1) > E_{\frac{1}{p_2},p_2}(\mathrm{Id},\S^1).
\]
Taking such $p_1, p_2$ we get
\[
 5\frac{4\pi^2}{2^{2-p_1}} < 5\frac{4\pi^2}{2^{2-p_2}}
 \le
 E_{\frac{1}{p_2}, p_2}(\mathrm{Id},\S^1) + \delta < E_{\frac{1}{p_1}, p_1}(\mathrm{Id},\S^1) + \delta.
\]
We get therefore a contradiction for every $p_1$ such that
\[
 5\frac{4\pi^2}{2^{2-p_1}} \ge E_{\frac{1}{p_1}, p_1}(\mathrm{Id},\S^1).
\]
The smallest value $p'$ for which it holds is therefore such that
\[
 5\frac{4\pi^2}{2^{2-p'}} = E_{\frac{1}{p'}, p'}(\mathrm{Id},\S^1).
\]
This equality can be written equivalently (see lemma \ref{decreasing_energy}) as
\[
 5\frac{4\pi^2}{2^{2-p'}} = 2^{p'} \pi \int_{0}^{\pi} \left( \sin \gamma \right)^{p' - 2} \dif \gamma,
\]
which gives
\[
 5 \pi = \int_{0}^{\pi} \left( \sin \gamma \right)^{p' - 2} \dif \gamma.
\]
We can check numerically that $p' \approx 1.13924$. Therefore, for $p \in [p' ,2]$ there must be a minimizer of degree 1.
\end{proof}
\appendix
\section{Border value of $p$}\label{a:border_p}

\begin{lemma}\label{decreasing_energy}
For $1 < p_1 < p_2 < 2$ there is
\[
  E_{\frac{1}{p_1},p_1}(\mathrm{Id},\S^1) > E_{\frac{1}{p_2},p_2}(\mathrm{Id},\S^1).
\]
\end{lemma}

\begin{proof}
It suffices to show that
\[
\begin{split}
     \frac{\partial}{\partial p} E_{\frac{1}{p}, p}(\mathrm{Id},\S^1)
\end{split}
\]
is negative on the interval $(1,2)$. Let $p \in (1,2)$. Using polar coordinates, we obtain
\begin{align*}
     E_{\frac{1}{p}, p}(\mathrm{Id},\S^1)
     &= \int_{\S^1} \int_{\S^1} |x-y|^{p - 2} \dx \dy \\
     &= \int_0^{2\pi} \int_0^{2\pi} \brac{(\cos \alpha - \cos \beta)^2 + (\sin \alpha -\sin \beta)^2}^{\frac{p-2}{2}} \dif \alpha \dif\beta\\
     &= \int_0^{2\pi} \int_0^{2\pi}\brac{4\sin^2\brac{\frac\alpha2 - \frac\beta2}}^{\frac{p-2}{2}} \dif \alpha \dif \beta\\
     &= \int_0^{2\pi} \int_0^{2\pi} \abs{2 \sin\brac{\frac\alpha2 - \frac\beta2}}^{p - 2} \dif \alpha \dif \beta \\
     &= \int_0^{2\pi} \int_{-\beta}^{2\pi-\beta} \abs{2 \sin\brac{\frac\gamma2 }}^{p - 2} \dif \gamma \dif \beta, \\
     & = 2 \int_0^{2\pi} \int_{-\beta/2}^{\pi-\beta/2} \abs{2 \sin \gamma}^{p - 2} \dif \gamma \dif \beta.
\end{align*}
Moreover,
\begin{align*}
     \int_{-\beta/2}^{\pi-\beta/2} \abs{2 \sin \gamma}^{p - 2} \dif \gamma
     =
     \int_{0}^{\pi} \abs{2 \sin \gamma}^{p - 2} \dif \gamma
     = 2^{p-1} \int_{0}^{\frac{\pi}{2}} \left( \sin \gamma \right)^{p - 2} \dif \gamma.
\end{align*}
Thus,
\begin{align*}
     E_{\frac{1}{p}, p}(\mathrm{Id},\S^1)
     = 2^{p+1} \pi \int_{0}^{\frac{\pi}{2}} \left( \sin \gamma \right)^{p - 2} \dif \gamma.
\end{align*}
Putting $w = \sin^2 \gamma$, we get
\begin{align*}
     E_{\frac{1}{p}, p}(\mathrm{Id},\S^1)
     = 2^{p+1} \pi \int_{0}^{1} w^{\frac{p - 2}{2}} \frac{\dif w}{2 \sqrt{w} \sqrt{1-w}}
     &= 2^{p} \pi \cdot B\left( \frac{p-1}{2}, \frac{1}{2} \right),
\end{align*}
where $B$ is the Euler beta function. Therefore,
\begin{align*}
     \frac{\partial}{\partial p} E_{\frac{1}{p}, p}(\mathrm{Id},\S^1)
     & = 2^{p} \log(2) \pi \cdot B\left( \frac{p-1}{2}, \frac{1}{2} \right)
     + 2^{p} \pi \cdot \frac{\partial}{\partial p} \left( B\left( \frac{p-1}{2}, \frac{1}{2} \right) \right).
\end{align*}
It is well-known that
\begin{align*}
     \frac{\partial}{\partial z_1} \left( B\left( z_1, z_2 \right) \right)
     =
     B\left( z_1, z_2 \right) \left( \phi(z_1) - \phi(z_1 + z_2) \right),
\end{align*}
where $\phi(z)$ is the digamma function. Hence, we obtain
\begin{align*}
     \frac{\partial}{\partial p} E_{\frac{1}{p}, p}(\mathrm{Id},\S^1)
     & = 2^{p-1} \pi \cdot B\left( \frac{p-1}{2}, \frac{1}{2} \right) \cdot \left( 2 \log 2 + \phi\left( \frac{p-1}{2} \right) - \phi\left( \frac{p}{2} \right) \right).
\end{align*}
As the Euler beta function is positive for real positive arguments, the derivative of the energy is negative if and only if
\begin{align*}
     2 \log 2 + \phi\left( \frac{p-1}{2} \right) - \phi\left( \frac{p}{2} \right) < 0.
\end{align*}
Now, we use the series expansion of the $\phi$ function, which can be found in \cite[p. 259]{abramowitz}. We get that for $z \notin \mathbb{N}$, there is
\begin{align*}
     \phi\left( z \right) = - \gamma + \sum_{n=0}^{\infty} \frac{z-1}{(n+1)(n+z)},
\end{align*}
where $\gamma$ is the Euler-Mascheroni constant, and the series converges for all $z > 0, z \notin \mathbb{N}$. As $\frac{p-1}{2}, \frac{p}{2} \notin \mathbb{N}$, we obtain
\begin{align*}
     \phi\left( \frac{p-1}{2} \right) - \phi\left( \frac{p}{2} \right)
     & = \sum_{n=0}^{\infty} \frac{\frac{p-1}{2}-1}{(n+1)\left(n+\frac{p-1}{2} \right)} - \sum_{n=0}^{\infty} \frac{\frac{p}{2}-1}{(n+1) \left(n+\frac{p}{2} \right)} \\
     & = \sum_{n=0}^{\infty} \frac{1}{n+1} \cdot \frac{(2n+p)(p-3) - (2n+p-1)(p-2)}{(2n+p-1)(2n+p)} \\
     & = - 2 \sum_{n=0}^{\infty} \frac{1}{(2n+p-1)(2n+p)}.
\end{align*}
Hence, the derivative of the energy is negative if and only if
\begin{align*}
     \log 2 < \sum_{n=0}^{\infty} \frac{1}{(2n+p-1)(2n+p)}.
\end{align*}
Now, as every summand $\frac{1}{(2n+p-1)(2n+p)}$ is a decreasing function of $p$ for $p \in (1,2)$, so is the whole series. Therefore, it suffices to show that the weak version of the inequality holds for $p=2$. However, we obtain exactly
\begin{align*}
     \sum_{n=0}^{\infty} \frac{1}{(2n+1)(2n+2)} = \log(2).
\end{align*}
This finishes the proof of the lemma.
\end{proof}

\bibliographystyle{abbrv}%
\bibliography{bib1}%

\begin{thebibliography}{1}

\bibitem{abramowitz}
M.~Abramowitz and I.~A. Stegun.
\newblock {\em Handbook of mathematical functions with formulas, graphs, and
  mathematical tables}.
\newblock National Bureau of Standards Applied Mathematics Series, No. 55. U.
  S. Government Printing Office, Washington, DC, 1964.
\newblock For sale by the Superintendent of Documents.

\bibitem{BBM}
J.~Bourgain, H.~Brezis, and P.~Mironescu.
\newblock Lifting, degree, and distributional {J}acobian revisited.
\newblock {\em Comm. Pure Appl. Math.}, 58(4):529--551, 2005.

\bibitem{BL17}
L.~Brasco and E.~Lindgren.
\newblock Higher {S}obolev regularity for the fractional {$p$}-{L}aplace
  equation in the superquadratic case.
\newblock {\em Adv. Math.}, 304:300--354, 2017.

\bibitem{BM-book}
H.~Brezis and P.~Mironescu.
\newblock {\em Sobolev maps to the circle---from the perspective of analysis,
  geometry, and topology}, volume~96 of {\em Progress in Nonlinear Differential
  Equations and their Applications}.
\newblock Birkh\"{a}user/Springer, New York, [2021] \copyright 2021.

\bibitem{Lindqvist}
P.~Lindqvist.
\newblock {\em Notes on the {\em p}-Laplace Equation}.
\newblock Number 102 in University of Jyväskylä Department of Mathematics and
  Statistics Report. University of Jyväskylä, second edition, 2017.

\bibitem{mazowiecka2020minimal}
K.~Mazowiecka and A.~Schikorra.
\newblock Minimal {$W^{s,\frac ns}$}-harmonic maps in homotopy classes.
\newblock {\em J. Lond. Math. Soc. (2)}, 108(2):742--836, 2023.

\bibitem{Mazowiecka_Schikorra_2024}
K.~Mazowiecka and A.~Schikorra.
\newblock {${s}$-Stability for ${W}^{s,n/s}$-Harmonic Maps in Homotopy Groups}.
\newblock {\em Annales de l'Institut Henri Poincaré C, Analyse Non Linéaire},
  2024.
\newblock Published online first.

\bibitem{M07}
P.~Mironescu.
\newblock Sobolev maps on manifolds: degree, approximation, lifting.
\newblock In {\em Perspectives in nonlinear partial differential equations},
  volume 446 of {\em Contemp. Math.}, pages 413--436. Amer. Math. Soc.,
  Providence, RI, 2007.

\bibitem{Mironescu}
P.~Mironescu.
\newblock Profile decomposition and phase control for circle-valued maps in one
  dimension.
\newblock {\em C. R. Math. Acad. Sci. Paris}, 353(12):1087--1092, 2015.

\end{thebibliography}

\end{document}